\documentclass[11pt]{article}
\usepackage{amsfonts, amssymb, amsthm, amsmath}

\newtheorem{prop}{Proposition}[section]

\newtheorem{thm}[prop]{Theorem}
\newtheorem{cor}[prop]{Corollary}

\newtheorem{rem}[prop]{Remark}
\newtheorem{stuff}{\rm \hspace{-4pt}}[subsection]
\newcommand{\grass}{{\rm Grass}}
\newcommand{\chow}{{\rm Chow}}

\begin{document}

\title{On the finiteness theorem for rational maps\\
on a variety of general type}
\author{Lucio Guerra \and Gian Pietro Pirola}
\date{}

\maketitle

\begin{abstract}
{\noindent The dominant rational maps of finite degree from a fixed variety 
to varieties of general type, up to birational isomorphisms, form a finite set. 
This has been known as the Iitaka-Severi conjecture, and is nowdays 
an established result, in virtue of some recent advances in the theory of 
pluricanonical maps. We study the question of finding some effective estimate 
for the finite number of maps, and to this aim we provide some update and refinement 
of the classical treatment of the subject. \medskip


  }
\end{abstract}

\section*{Introduction}

Let $X$ be an algebraic variety of general type, over the complex field. 
The dominant rational maps of finite degree
$X \dasharrow Y$ to varieties of general type, 
up to birational isomorphisms $Y \dasharrow Y'$,
form a finite set. This has been known as the Iitaka-Severi conjecture,
and is nowdays an established result. 
The approach introduced by Maehara \cite{M}
obtains the full range of applicability in virtue of some recent
advances in the theory of pluricanonical maps, 
originating with the work of Siu \cite{S1}, \cite{S2}, 
and successively due to Tsuji \cite{Tsu},
Hacon and McKernan \cite{HK}, and Takayama \cite{Tak}, \cite{Tak2}. 
  
In this paper we study the question of finding some effective
estimate for the finite number of maps in the theorem,
in the same line as works of
Catanese \cite{Cat}, Bandman and Dethloff \cite{BD},
Heier \cite{H}, and
the article \cite{Gu} by the first author.
It is worth mentioning that 
there is another approach to the question,
developed by Naranjo and the second author in \cite{NP},
in a slightly less general setting.
To the aim of obtaining effective results, we moreover provide 
some update and refinement of the classical treatment
of the subject,
according to the recent progress.

The approach consists, as usual in this kind of
questions, of two main steps,
that may be called: rigidity and boundedness,
from which the finiteness theorem follows.
We shortly describe the main points that represent the
contribution of the present paper.

We bring the rigidity theorem to a general form (theorem \ref{rigidity}), 
avoiding certain technical restrictions that are in \cite{M},
and relying on \cite{HK} and \cite{S1}.
We point out that generic rigidity extends to limits
(proposition \ref{primitive}).
We show that bounds for the degree of a map and of its graph 
(see \S \ref{boundedness}) are naturally obtained in terms of the 
canonical volume, the new invariant 
arising in the asymptotic theory of divisors, cf. Lazarsfeld's book \cite{L}. 
Finally we discuss the structure of the proof of the finiteness theorem.
We propose a new argument,  relating the finite number of rational  maps
and the irreducible components of a certain bunch of
(subvarieties of) Chow varieties
(theorem \ref{finite number}),
in the perspective of obtaining an effective bound
in terms of the complexity of these Chow varieties
(see proposition \ref{effective}). 
Also in this part we use some recent powerful results of minimal model
theory, namely \cite{HK} and \cite{Tak2}.
The computations of complexity are explained in \S \ref{chowcomplexity}.
\bigskip

\small \noindent {\em Acknowledgements.}
The first author is partially supported by:
Finanziamento Ricerca di Base 2007 Univ. Perugia.
The second author is partially supported by:
1) PRIN 2007 {\em ``Spazi di moduli e teorie di Lie"}; 
2) Indam (GNSAGA);
3) FAR 2006 (PV):{\em ``Variet\`{a} algebriche, calcolo
algebrico, grafi orientati e topologici"}.
\normalsize 

\section{Results on pluricanonical maps}
We collect here the recent results in the theory of
pluricanonical  maps
that will be used in the paper.

\subsection{Volume of a big divisor} \label{vol}
Recall that the volume of a divisor is defined as
$$\mbox{\rm vol}_X(D) = \lim_m \; h^0(X,mD)/(m^n/n!)$$ 
and for a big divisor also there is the description 
$$\mbox{\rm vol}_X(D) = \lim_m \; (mD)^{[n]}/m^n$$ 
in terms of moving self-intersection
numbers $D^{[n]}$ (cf. Lazarsfeld \cite{L},  II.11.4, p. 303). 

The following observation will be useful:
if $X \dasharrow X' \subset \mathbb P^{m}$ is a birational embedding,
and $D$ is the pullback on $X$ of the general hyperplane of $\mathbb P^{m}$, then
$$\deg X' \leq \mbox{\rm vol}_X(D)$$
(cf. \cite{HK}, p. 5).

\subsection{Pluricanonical embeddings} \label{plumap}
A recent achievement in the theory of pluricanonical maps is the following
theorem of uniform pluricanonical birational embedding, cf.
Tsuji \cite{Tsu},  Hacon-McKernan \cite{HK},
Takayama \cite{Tak}.
\begin{thm} \label{HKT}
For any dimension $n$, there is some positive integer $r_n$ such that: for
every $n$-dimensional
variety $V$ of general type the multicanonical divisor
$r_nK_V$ defines a birational embedding $V \dashrightarrow V' \subset
\mathbb P^M$.
\end{thm}

\begin{rem} \em
For the theorem we also have a bound 
$$\deg V' \leq  d_{V},$$  
arising from \S \ref{vol}, take $d_{V} = {\rm vol} (r_{n}K_{V})$.
Moreover from elementary geometry 
we also have a bound  $$M \leq M_{V},$$ 
take for instance $M_{V} = d_{V}+n-1$.
\end{rem}

\subsection{Pluricanonical maps in a family} 
Let $q: Y \rightarrow T$ be a projective morphism of nonsingular varieties. Consider
the relative pluricanonical bundle $\omega_{Y/T}^{\otimes m}$ on $Y$, and
the coherent sheaf $q_*(\omega_{Y/T}^{\otimes m})$ on $T$.
If $q$ is a smooth morphism,
the restriction of $\omega_{Y/T}^{\otimes m}$ to the fibre $Y_t$ 
is the pluricanonical bundle $\omega_{Y_t}^{\otimes m}$,  
and the  natural map 
$q_*(\omega_{Y/T}^{\otimes m}) \otimes k(t) \rightarrow
H^0(Y_t,\omega_{Y_t}^{\otimes m})$
is an isomorphism for general $t \in T$. 
Consider the induced rational map of families
$$Y \dasharrow {\mathbb P}(q_*(\omega_{Y/T}^{\otimes m})).$$ 
If $q$ is a smooth morphism this is, generically,
the family of $m$-canonical maps of fibres.
In this setting the following
theorem of Kawamata \cite{K} holds:
\begin{thm} [\bf  semipositivity] \label{semipositivity}
The sheaf $q_*(\omega_{Y/T}^{\otimes m})$ is nef, for a family of
connected fibres, provided that $T$ is a curve.
\end{thm}

Another fundamental property of relative pluricanonical sheaves 
is the following 
theorem of Siu, see \cite{S1} and \cite{S2}:
\begin{thm}[\bf invariance of plurigenera]  \label{Siu}
In a smooth projective family every plurigenus is constant. \end{thm}

In the theory of pluricanonical maps,
the need to allow singular varieties arises,
and an outstanding role is played by the
canonical singularities, cf.  \cite{R}.
For a possibly singular variety,
the plurigenus of any resolution of singularities is
independent of the resolution, and is called the
plurigenus of the variety.

In this general situation, one still has the following results.
In a  projective family of varieties, over a smooth curve, 
every plurigenus is a lower semicontinuous function;
if the varieties in the family have canonical singularities,
then every plurigenus is constant;
see Takayama \cite{Tak2}. 
The invariance of plurigenera is also known
for flat families (of any dimension)
of canonical singularities of general type,
see Kawamata \cite{K2}.

Moreover we point out 
the following result, for which a proof will be
given in the next section.

\begin{thm}[\bf invariance of the general type] \label{generaltype}
In a  projective family of varieties, over a smooth curve, 
the varieties of general type form an open subfamily,
in the Zariski topology.
\end{thm}

\subsection{Extension of differentials}

The results in the previous section
on the variation of plurigenera are all based on 
the technique of extension of differentials,
from a special fibre to the total space of the family.
The following general result, for one-dimensional families,
is due to Takayama \cite{Tak2},
we quote it in a slightly less general form,
suitable for our purposes.

\begin{thm} \label{extension}
Let $\pi:V \rightarrow S$ be a projective morphism
of nonsingular varieties, with $S$ a curve,
and let $Z$ be an irreducible component of some fibre $V_{0} = \pi^{-1}(t_{0})$,
also a nonsingular variety.
Assume moreover that $\pi$ has connected fibres.
Then  the restriction map 
$$\pi_{*}\mathcal O_{V}(m K_{V})_{\,t_{0}} \ \longrightarrow \
H^{0}(Z,\mathcal O_{Z}(mK_{Z}))$$
is surjective for every  $m>0$.
\end{thm}

Using the extension theorem, we may
give a proof of theorem \ref{generaltype}.

\begin{proof}
Let $T$ be a nonsingular curve, and
let $q: Y \rightarrow T$ be a family of projective varieties,
i.e. a projective morphism with reduced irreducible fibres $Y_t$ 
of the same dimension, say $n$.
We have to show that:
if some fibre $Y_{0}$ is of general type then
the nearby fibres $Y_{t}$ are of  general type.

Consider some resolution of singularities
$\mu: V \rightarrow Y$ such that the strict transform $Z = Y_{0}'$
is smooth.  So $Z$ is of general type. Then 
$\dim H^{0}(Z,mK_{Z}) \geq cm^{n}$ for $m \gg 0$.

From theorem  \ref{extension} we have that
the restriction homomorphism
$$\begin{array}{ccc}
\pi_{*}\mathcal O_{V}(mK_{V})_{\,t_{0}} & \longrightarrow & 
H^{0}(Z, mK_{Z}) 
\end{array}$$
is surjective. 

The image $\pi_{*}\mathcal O_{V}(mK_{V})$ is a 
torsion free coherent sheaf on the smooth curve $T$,
hence it is a locally free sheaf.
So the dimension of 
$\pi_{*}\mathcal O_{V}(mK_{V}) \otimes k(t)$
is constant.
In $t_{0}$ this dimension  is  $\geq cm^{n}$, 
by what we have seen above.
\newcommand{\localmentelibero}
{let  $f:Y \rightarrow S$ be flat,  $\mathcal F$ on $Y$ be flat over $S$.
if $f_{*} \mathcal{F} \otimes k(t) \rightarrow H^{0}(Y_{t}, \mathcal{F}_{t})$
is surjective then it is an isomorphism, and the same holds in a neighborhood of $t$.
moreover $f_{*} \mathcal{F}$ is locally free in a neighborhood of $t$
[Hartshorne,  p. 290] }

Moreover as the composite map
$V \rightarrow Y \rightarrow T$ is generically smooth
we may assume that
in some neighborhood of $t_{0}$ for every $t \neq t_{0}$
the induced map $V_{t} \rightarrow Y_{t}$ is a
resolution of singularities.
Since $mK_{V}|_{V_{t}} = m K_{V_{t}}$
we have the inclusion
$$\pi_{*} \mathcal O_{V}(mK_{V}) \otimes k(t)  \hookrightarrow 
H^{0}(V_{t}, \mathcal O_{V_{t}}(mK_{V}|_{V_{t}}))
= H^{0}(V_{t}, mK_{V_{t}}).$$
It follows that
$\dim H^{0}(V_{t},mK_{V_{t}}) 
\geq cm^{n}$ for $m \gg 0$,  hence  $Y_{t}$
is of general type. This holds for every $t$
in a neighborhood of $t_{0}$.
\end{proof}

\section{Rigidity}

First step is a theorem of rigidity, which is more precisely a theorem
of birational triviality of families.

\subsection{The rigidity theorem}

\begin{thm} \label{rigidity}
Let $X$ be a projective variety of general type, of dimension $n$. \linebreak
Let $T$ be a smooth variety, 
and let $Y \rightarrow T$ be a smooth  family of 
$n$-dimensional projective varieties $Y_t$ of
general type. Assume that
$f: X\times T \dasharrow Y$ is a family of  dominant rational maps, i.e.
that $f$ restricts to a dominant rational map
$f_t: X \dasharrow Y_t$ for every $t$. 
Then there is $g: Y \dasharrow Y_0 \times T$, a
birational isomorphism over $T$, defined on every $Y_{t}$, such that
$g \circ f$ is  a trivial family $f_0 \times 1: X \times T \dasharrow Y_0 \times T$. 
Therefore all maps $f_t$ are birationally equivalent.
\end{thm}

This is a slightly more general version
than the one in Maehara \cite{M}, 
avoiding some technical restrictions. 
We recall here that two rational maps $f: X \dasharrow Y$ and
$f': X \dasharrow Y'$ are said to be birationally equivalent if
there is a birational isomorphism $g: Y \dasharrow Y'$ such that
$f' = g \circ f$. The proof of the theorem will be given
in the next section.
We point out the following:

\begin{cor} \label{weakrigidity} 
Let $X$ be a projective variety of general type, of dimension $n$. \linebreak
Let $T$ be a smooth variety, 
and let $Y \rightarrow T$ be a  family of 
$n$-dimensional projective varieties $Y_t$ of
general type. Assume that
$f: X\times T \dasharrow Y$ is a family of  dominant rational maps
$f_t: X \dasharrow Y_t$. Then 
almost all maps $f_t$ are birationally equivalent.
\end{cor}

\begin{proof}
Choose some resolution of singularities $Z \rightarrow Y$.
By generic smoothness, over some open subset $T' \subset T$
the family $Z \rightarrow T$ is smooth and 
the fibre $Z_{t}$ is a (nonsingular) birational model of $Y_{t}$.
The assertion follows from the theorem
applied to the restricted family $Z' \rightarrow T'$.
\end{proof}

\begin{rem} \label{extended}  \em
We think that the rigidity theorem holds more generally if
$Y \rightarrow T$ is a  family of projective varieties $Y_t$,
with canonical singularities and of general type. 
\end{rem}

\subsection{Proof of the rigidity theorem}
\begin{stuff} \em
The variation of higher differentials  (cf. \cite{M} \S 4.1). \label{variation}
Let $f: X\times T \dasharrow Y$ be a family of rational maps, as in the statement. 
Let $m = r_{n}$ be the integer defined in theorem \ref{HKT}, and
consider the exact sequence of  sheaves on $T$ 
$$0 \rightarrow q_*(\omega_{Y/T}^{\otimes m}) \rightarrow  V \otimes
{\mathcal O}_T \rightarrow {\mathcal Q} \rightarrow 0
\ \mbox{ where } \  V := H^0(X,\omega_X^{\otimes m})$$ 
in which we know from theorem \ref{Siu} that 
$q_*(\omega_{Y/T}^{\otimes m})$ is locally free. 
Let $\nu$ be the rank of this sheaf.

The claim is that the image of this inside $V \otimes {\mathcal O}_T$
is a trivial subsheaf $W \otimes {\mathcal O}_T$.
It follows that the composite map
$X \times T \dasharrow Y \dasharrow {\mathbb P}(W) \times T$ 
is of the form $f_{0} \times 1$, which gives the theorem.  

In order to prove the claim, we may restrict to a Zariski open subset of $T$. 
{Suppose indeed that for some $T' \subset T$
the isomorphism
$q_*(\omega_{Y/T'}^{\otimes m}) 
\overset{\sim}{\rightarrow} W \otimes {\mathcal O}_{T'}$
is proved. 
Since $q_*(\omega_{Y/T}^{\otimes m})$ is locally free
it follows that there is an induced map
$q_*(\omega_{Y/T}^{\otimes m}) 
\rightarrow W \otimes {\mathcal O}_{T}$,
and this is injective as a map of bundles, of the same rank,
hence it is an isomorphism.}

Therefore we may assume that the cokernel $\mathcal Q$ is locally free.
Then the exact sequence above may be seen
as the pullback of the universal sequence on the suitable Grassmannian through the morphism 
$$T \rightarrow \grass_\nu(V) \mbox{ \ sending } t \mbox{ to }
H^0(Y_t,\omega_{Y_t}^{\otimes m}).$$ 

{\bf Claim:}  this map  is constant. \medskip
\\ 
In order to prove the claim, it is enough to restrict to curves in $T$.
\end{stuff}
\begin{stuff} \em
One parameter families  (cf. \cite{M}  \S 4.2).
Let $C $ be a nonsingular curve, let $Y$ be a nonsingular variety,
and assume that $Y \rightarrow C$ is
a smooth projective family. If 
$C \hookrightarrow D$ is an embedding into a
projective nonsingular curve, 
complete the family to a flat family $Y' \rightarrow D$. 
Some limit fibre is possibly singular.
Then apply {\em semistable reduction}: there is  a projective nonsingular
curve $B$ with a finite map $\varphi : B \rightarrow D$, 
and there is a projective nonsingular variety $Z$
with a morphism $Z \rightarrow B$, and with a morphism of families 
$Z \rightarrow Y'$, such that: 
restricted to $A= \varphi^{-1}(C)$ 
the family coincides with the pullback of the original family over $C$, and moreover the new 
limit fibres  have only normal crossings as singularities. 
$$\begin{array}{ccccc}
Y & \hookrightarrow & Y' & \leftarrow & Z \\  
\downarrow && \downarrow && \downarrow \\ 
C & \hookrightarrow & D & \underset{\varphi}{\leftarrow} & B
\end{array}$$

Let $X\times C \dasharrow Y$ be a family of rational maps.
It is viewed as a family $X\times D \dasharrow Y'$, 
and can be lifted to a family $X\times B \dasharrow Z$.

As the construction of the relative canonical sheaf is compatible with pullback, one shows that
the induced map $A \rightarrow \grass_\nu(V)$,
to the appropriate Grassmannian, factorizes as $A \rightarrow C$
followed by the induced map $C \rightarrow \grass_\nu(V)$. 
So it is enough to prove that 
the map on $A$ is constant
\end{stuff}

\begin{stuff} \em
Proof of the claim  (cf. \cite{M} \S 5).
Thus one is reduced to the
situation where $B$ is a nonsingular projective curve, $Z$ is a nonsingular
projective variety and $Z \rightarrow B$ 
is a family with connected fibres having only normal
crossings. Consider the
exact sequence $0 \rightarrow q_*(\omega_{Z/B}^{\otimes m}) \rightarrow  V
\otimes {\mathcal O}_B \rightarrow {\mathcal Q} \rightarrow 0$.
Let $A \subset B$ be an open subset such that the restriction ${\mathcal Q}_A$ is locally
free. 
Let $T({\mathcal Q}) \subset {\mathcal Q}$ be a torsion subsheaf such that ${\mathcal
Q}/T({\mathcal Q})$ is locally free. Then consider the 
exact sequence 
$$0 \rightarrow {\mathcal K} \rightarrow  V \otimes
{\mathcal O}_B \rightarrow {\mathcal Q}/T({\mathcal Q}) \rightarrow 0$$ 
in which the kernel ${\mathcal K}$ is locally free. 
Note that $q_*(\omega_{Z/B}^{\otimes m}) \subset {\mathcal K}$ 
and the two sheaves coincide on $A$. It follows that there is an induced morphism
\begin{center}
$B \rightarrow \grass_\nu(V)$ \ {sending} $t$ to  ${\mathcal K}_t$
\end{center}
which is
$H^0(Z_t,\omega_{Z_t}^{\otimes m})$ for $t \in A$.

The pullback of the Pl\"ucker line bundle is $\det {\mathcal K}^\vee$. If
this is a non-constant map then
$\deg {\mathcal K}^\vee > 0$ and $\deg {\mathcal K} < 0$. 

On the other hand from the semipositivity theorem \ref{semipositivity}
we have that the sheaf $q_*(\omega_{Z/B}^{\otimes m})$ is nef.
Then $\det q_*(\omega_{Z/B}^{\otimes m})$ is nef and $\deg
q_*(\omega_{Z/B}^{\otimes m}) \geq 0$.
Moreover $q_*(\omega_{Z/B}^{\otimes m}) \subset {\mathcal K}$ and the two
coincide on
$A$, and this implies $\deg {\mathcal K} \geq \deg
q_*(\omega_{Z/B}^{\otimes m}) \geq 0$,
contradiction.
\end{stuff}

\subsection{Generic rigidity and limit maps}

Let $X$ be a projective variety, of dimension $n$,
let $Y \rightarrow T$ be a  family of 
$n$-dimensional projective varieties,
with $T$ a smooth curve, and let
$f: X\times T \dasharrow Y$ be a 
family of  dominant rational maps $f_t: X \dasharrow Y_t$,
defined over all of $T$. 
The graph $F:= \Gamma(f)$
with the natural projection $F \rightarrow T$
represents the family
of graphs $F_{t}=\Gamma(f_{t})$.
Moreover $\deg (f_{t}) =:d$ is  constant in the family.

A more general definition is obtained if we admit that
some special fibre $Y_{t_{0}}$ may have multiplicity,
i.e. may be of the form $d_{0}Y_{0}$, multiple of an irreducible.
In this situation $\deg (f_{t}) =:d$ is  constant for $t \neq t_{0}$.
Since $\pi_{*} \Gamma(f_{t}) = d Y_{t}$ holds for $t \neq t_{0}$,
where $\pi$ denotes projection to $Y$, the same holds for every $t$,
and $\pi_{*} \Gamma(f_{t_{0}}) = dd_{0} Y_{0}$ implies that
$\deg(f_{t_{0}}) = dd_{0}$. So in particular
the limit map has higher degree.

The more general definition above may be expressed equivalently
by requiring that the graph $\Gamma(f)$ defines 
a family of graphs $\Gamma(f_{t})$, without requiring
that $Y$ defines the family of images, i.e. that every
fibre $Y_t$ is reduced.

Now consider a family which is {\em generically trivial},
as in  corollary \ref{weakrigidity},
i.e. which  is obtained as 
$$X \times T \overset{h \times 1}{\dasharrow} V \times T 
\overset{g}{\dasharrow} Y$$
a constant family 
followed by a birational isomorphism $g$,
defined over $T - \{t_{0}\}$.
In this situation we have $f_{t} = g_{t} \circ h$ for $t \neq t_{0}$,
so all these maps are birationally equivalent,
of  degree $d:= \deg(f_{t}) = \deg(h)$.
Concerning the limit map $f_{t_{0}}$
we can say the following.

\begin{prop}  \label{primitive}
Assume that the family $f$ is generically trivial
and of constant degree, i.e. that
$\deg(f_t)=d$ for every $t$.
Then the limit map $f_{t_{0}}$ is 
in the same birational equivalence class
as the general $f_{t}$
\end{prop}

\begin{proof}
We have 
the graphs $F_{t} = \Gamma(f_{t})$ for every $t$, and 
$G_{t} = \Gamma(g_{t})$ for $t \neq t_{0}$,
and  $H = \Gamma(h)$. 
In terms of composition of correspondences
we have that $F_{t} \subset  H \circ G_{t}$ 
holds for $t \neq t_{0}$
(a little conflict of notations is hidden here).

The family $G_{t}$ converges to a cycle $G_{0}$ in the space of cycles
of $V \times Y$.
The intersection number $G_{t} \cdot (\{x\}\times Y) =1$
is constant in the deformation, and 
$G_{0} \cdot (\{x\}\times Y) =1$ implies that the part of $G_{0}$
that dominates $V$ is an irreducible reduced cycle and is the graph $G'_{0}$ of some
rational map $g_{0}: V \dasharrow Y_{0}$. By continuity
we have that $F_{t_{0}} \subset H \circ G_{0}$,
and more precisely that $F_{t_{0}} \subset H \circ G'_{0}$,
since $F_{0}$ dominates $X$.
But this implies that $f_{t_{0}} = g_{0} \circ h$.
As the degree is constant in the family, then
$g_{0}$ is birational, hence $f_{t_{0}}$ is birationally
equivalent to $h$ and hence to every $f_{t}$. 
\end{proof}

\begin{cor} \label{globalrigidity} 
Let $X$ be a projective variety of general type, of dimension $n$. \linebreak
Let $T$ be a smooth curve, 
and let $Y \rightarrow T$ be a  family of 
$n$-dimensional projective varieties $Y_t$ of
general type. Assume that
$f: X\times T \dasharrow Y$ is a family of  dominant rational maps
$f_t: X \dasharrow Y_t$, of constant degree. Then 
all maps $f_t$ are birationally equivalent.
\end{cor}

\begin{proof}
Immediate from corollary \ref{weakrigidity}
and proposition \ref{primitive} above.
\end{proof}

\begin{rem} \em  
There is a property of the limit map which forces the
generically trivial family to be of constant degree.
Let us say that
a dominant rational map of finite degree $X \dasharrow Y$
is {\em  primitive} if 
it admits no factorization  $X \dasharrow Z \dasharrow Y$
in which  both factor maps have degree $>1$.
The same argument as in the proof above implies that,
if the generic degree is $>1$, and if the limit map
$f_{t_{0}}$ is a primitive map, 
then it is 
in the same birational equivalence class
as the general $f_{t}$.
\end{rem}

\section{Boundedness} \label{boundedness}
Second step, that on a given variety the family of rational maps,
 of the type being considered, is a bounded family.
Here we point out that the canonical volume is a natural tool.

\subsection{Bound for  the graph of a map} \label{bdg}
Let $X$ be a nonsingular variety of general type, of dimension $n$.
Choose an embedding $X \hookrightarrow \mathbb P^N$ or more generally
a birational embedding
$$X \dashrightarrow X' \subset \mathbb P^N$$
and let $H$ be the  pullback of a general hyperplane.
Every rational map of finite degree $X \dasharrow Y$
which dominates a variety of general type,
followed by the $r_{n}$-canonical 
birational embedding 
$Y \dasharrow Y' \subset \mathbb P^{M}$, 
may be seen as a rational map
$X' \dasharrow Y' \subset \mathbb P^{M}$,
with bounded embedding dimension $M \leq M_{X}$, 
see \S \ref{plumap}.
This map has an associated  graph 
$$\Gamma \subset X'  \times \mathbb P^M$$
and the degree of the graph is bounded too.

\begin{prop} \label{grafico}
Let $X$ and $H$ be as in the setting above. There is some  positive
integer $\gamma_X$ such
that every associated graph  $\Gamma$ in $\mathbb P^N \times \mathbb P^M$
has bounded degree:
$$\deg \Gamma \leq \gamma_X.$$ 
More precisely we have:
\begin{enumerate}
\item $\deg \Gamma \leq \mbox{\rm vol}_{X}(H + r_{n}K_X)$ \ in general; 
\item $\deg \Gamma  \leq  (H + r_{n}K_X)^n$  \ if $H + r_{n}K_X$  is nef;
\item $\deg \Gamma \leq (2r_{n})^n \,\mbox{\rm
vol}(K_X)$ \  if $H \equiv r_{n}K_X$.
\end{enumerate}
\end{prop}
\begin{proof}
1. Consider the birational embedding
$X \dasharrow \Gamma \subset \mathbb P^N \times \mathbb P^M$ 
and apply \S \ref{vol}. It follows that 
$\deg \Gamma \leq \mbox{vol}_{X}(H + f^*r_{n}K_Y)$,
and this is $\leq \mbox{vol}_{X}(H + r_{n}K_X)$
as 'volume increases in effective directions': 
${\rm vol}_{X}(D) \leq {\rm vol}_{X}(D')$ if $D \leq D'$.
2. This is because ${\rm vol}_X(D) = D^{n}$ if $D$ is big and nef. 
3. is just a special case of point 1, using ${\rm vol}_X(rD) = r^{n}{\rm vol}_X(D)$.
\end{proof}

The proposition implies that every rational map on $X$,
of the type being considered, is birationally equivalent
to some map whose graph is a point in the disjoint union
of Chow varieties
$\bigsqcup\, \chow_{n,\gamma}(X' \times \mathbb P^{M_{X}})$,
taken over all $\gamma \leq \gamma_{X}$,  and
more precisely a point in the Zariski open subset
parametrizing rational maps of finite degree.

\subsection{Remarks on the degree of a map}
The degree of the graph is greater than the degree of the map.
We insert here a few remarks on how to
bound the degree of a map,
which however are not needed in the following.

\begin{prop} \label{grado}
Let $f : X \dasharrow Y$ be a rational map of finite degree.
\begin{enumerate}
 \item If $K_X$ and $K_Y$ are nef, then \
$\deg(f) \; K_Y^n \leq  K_X^n$;
\item if $K_X$ and $K_Y$ are big, then \
$\deg(f) \; \mbox{\rm vol} (K_Y) \leq \mbox{\rm vol} (K_X)$.
\end{enumerate}
\end{prop}

\begin{proof}
1. Use the asymptotic Riemann-Roch theorem
(cf. \cite{L}, I, 1.4.41): if $D$ is nef, then
$h^0(mD) = \frac{1}{n!} D^n\, m^n + O(m^{n-1})$. Since $f^*K_Y \leq K_X$
consider $h^0(m f^*K_Y)
\leq h^0(mK_X)$ and apply aRR, comparing the leading coefficients. 
2. Using the volume of a big divisor, expressed in terms of
moving self-intersection numbers, it is easy to see that  
under a rational map of finite degree one has
$\deg(f) \, \mbox{\rm vol}_Y (D) \leq \mbox{\rm vol}_X (f^*D)$. Moreover
$f^*K_Y \leq K_X$  implies
that $\mbox{\rm vol}_{X}(f^*K_Y) \leq \mbox{\rm vol} (K_X)$. 
\end{proof}

\begin{prop}
Let $X$ be a variety of general type. There is some positive 
rational number $\epsilon_X >0$ such that:
if $f : X \dashrightarrow Y$ is a rational map of finite degree which dominates $Y$ of general type,
then $\mbox{\rm vol} (K_Y) \geq \epsilon_X$  and therefore
$$\deg(f)  \leq \frac{1}{\epsilon_X} \; \mbox{\rm vol} (K_X).$$
\end{prop}

\begin{proof}
It has been shown by Hacon and McKernan (\cite{HK} cor.) that  for every $Y$ of
general type of given dimension  
one has  $\mbox{\rm vol} (K_Y)
\geq \epsilon$ for some  $\epsilon >0$. 
Therefore $\epsilon_X$ exists. Then use point 2
in the proposition above
\end{proof}

\section{Finiteness}
The finiteness theorem is the following

\begin{thm} \label{finiteness}
Let $X$ be a projective variety of general type, of dimension $n$. 
The set of rational maps $X \dasharrow Y$ 
which dominate $n$-dimensional projective varieties $Y$ of general type, 
up to birational isomorphisms $Y \dasharrow Y'$, is a finite set.
\end{thm}

We shortly discuss the original argument of \cite{M} and 
we propose a variation
which possibly leads to some effective bound for
the finite number of maps.

\subsection{The original argument}
The parameter space for rational maps of finite degree
$X \dasharrow Y' \subset \mathbb P^{M}$,
with fixed $M$, seen at the end of \S \ref{bdg},
is quasi-projective (and highly reducible).
Call it $T$ for a while.
There is an algebraic subset $Y \subset \mathbb P^{M} \times T$
whose reduced fibres in the projection over $T$ are
the dominated varieties $Y'$,
and there is a total rational map $X \times T \dasharrow Y$.

Then use {smoothening stratification:}
there is a stratification into (locally closed) 
smooth strata $T_{\alpha}$
and for each restricted family $Y_{\alpha} \rightarrow T_{\alpha}$
there is a resolution of singularities 
$Z_{\alpha} \rightarrow Y_{\alpha}$ such that
the composite map
$Z_{\alpha} \rightarrow T_{\alpha}$
is a smooth family.
So for every $t \in T$ also there is a 
resolution of singularities $Z_{t} \rightarrow Y_{t}$.
Moreover the stratification may be refined so that
the subset of $T$ parametrizing varieties
(admitting a smooth model) of general type
is a union of a subcollection of strata $T_{\beta}$,
thus a constructible subset.
cf. \cite{M}, \S 3.

It follows that:
{\em the number of equivalence classes of rational maps
in theorem {\em \ref{finiteness}} is 
smaller than the number of connected components of strata.} 
If two maps are in the same connected stratum,
there is a curve (possibly reducible)
connecting the two points and contained in the stratum. This gives
a family of rational maps $X \dasharrow Z_{t}$ such that
$Z_{t}$ is smooth of general type. 
Because of the rigidity theorem \ref{rigidity} they are all
birationally equivalent.

Unfortunately such a stratification seems to be rather intractable.
A variation of the argument is used by  Heier \cite{H}
for morphisms $X \rightarrow Y$ onto varieties with ample  $K_{Y}$. 
In this situation,
the observation is that
{the number of equivalence classes of these morphisms is smaller than
the number of connected components of the union of Chow varieties}.
We show in the next sections that the idea
is applicable in a more general situation.

\subsection{Parametrization}

Let $X$ be a nonsingular projective variety of general type,
of dimension $n$, and choose  a birational embedding
$X \dasharrow X' \subset \mathbb P^{N}$, as in \S \ref{bdg}. 

In the Chow variety
$\chow_{n,\gamma}(X' \times \mathbb P^{M})$  consider
the Zariski open subset
$$G_{\gamma}(X' \times \mathbb P^{M})$$
parametrizing graphs of rational maps $g: X' \dasharrow \mathbb P^{M}$
such that the (closed) image $Y' \subset \mathbb P^{M}$
still is of dimension $n$. For such a map we have
$$\deg_{2}\Gamma = \deg(g) \deg Y'  \mbox{ \ \ and \ \ }
p_{*}(\Gamma) = \deg(g) \; Y'.$$ 
For any $k>0$ with $\gamma \geq k$ define
$$G_{\gamma,k}(X' \times \mathbb P^{M})$$
by requiring that $\deg_{2}\Gamma = k$.
For any $d>0$ with $d|k$ define
$$G_{\gamma,k,d}(X' \times \mathbb P^{M})$$
by the condition that 
$p_{*} \Gamma$  { is of the form } $dY'$ (without
requiring that $Y'$ be irreducible).

\begin{prop} We have that:
\begin{enumerate}
\item
$G_{\gamma,k}(X' \times \mathbb P^{M})$
is a union of irreducible components of 
$G_{\gamma}(X' \times \mathbb P^{M})$,
\item
$G_{\gamma,k,d}(X' \times \mathbb P^{M})$
is an algebraic subset of $G_{\gamma,k}(X' \times \mathbb P^{M})$.
\end{enumerate}
\end{prop}

\begin{proof}
In a family of graphs $\Gamma_{t}$, parametrized by an
irreducible component of $G_{\gamma}(X' \times \mathbb P^{M})$,
the intersection number $\deg_{2}(\Gamma_{t})$ has a constant value $k>0$.
Thus $G_{\gamma,k}(X' \times \mathbb P^{M})$ is a union of
irreducible components of $G_{\gamma}(X' \times \mathbb P^{M})$.
This proves point (1).
The map $G_{\gamma,k}(X' \times \mathbb P^{M}) \longrightarrow 
\chow_{n,k}(\mathbb P^{M})$ such that $t \mapsto p_{*}(\Gamma_{t})$
is an algebraic correspondence. Hence the subset
defined by the condition that $p_{*}(\Gamma_{t})$ is of the form $dY'$,
i.e. that the associated Chow form is of the form $F^{d}$,
is the inverse image of a $d$-ple Veronese variety. This proves point (2).
\end{proof}

Let $m = r_{n}$ be the integer defined in theorem \ref{HKT}.
Define in $G_{\gamma}(X' \times \mathbb P^{M})$ the subset
$$S_{\gamma}(X' \times \mathbb P^{M}),$$
parametrizing maps 
such that the image $Y'$ is of general type
and $m$-canonically embedded,
and similarly define 
$$S_{\gamma,k,d}(X' \times \mathbb P^{M}).$$
Note that the degrees $\gamma,k,d$ are birationally invariant.

\subsection{Refined finiteness theorem}

We point out a more precise version of the finiteness theorem,
in which the finite number of maps is related to
the finite number of irreducible components
of a certain bunch of (subvarieties of) Chow varieties,
over which we have some control.

\begin{thm} \label{finite number}
Let $X$ be a projective variety of general type, of dimension $n$. 
The number of  rational maps on $X$
which dominate $n$-dimensional projective 
varieties of general type, 
up to birational equivalence,
is bounded by the number of 
{irreducible} components
of the disjoint union 
$$\bigsqcup_{} \;
G_{\gamma,k,d}(X' \times \mathbb P^{M_{X}})$$
taken over all $\gamma,k,d$ such that
$1< k \leq \gamma \leq \gamma_{X}$
 and $d>1$ is a divisor of $k$.
\end{thm}

\begin{proof} 
For every $M$ and $\gamma,k,d$,
there is a correspondence
$$S_{\gamma,k,d}(X' \times \mathbb P^{M})/_{\sim } \ \ 
\raisebox{5pt}{$\swarrow \raisebox{8pt}{...} \searrow$}  \ \
{i.c.}\; G_{\gamma,k,d}(X' \times \mathbb P^{M})$$
between the set of birational equivalence classes of maps and the set of 
irreducible components of the algebraic set,
in which a class of maps corresponds to an irreducible component
if and only if the two meet at some point.
Each equivalence class corresponds to one or several components,
and the key point is that: two different classes of  maps cannot
correspond to one and the same irreducible component of the algebraic set.
For $M = M_{X}$, this gives the theorem.

Assume that two points in $S_{\gamma,k,d}(X' \times \mathbb P^{M})$ 
are in the same irreducible component
of $G_{\gamma,k,d}(X' \times \mathbb P^{M})$. Then they
are connected by some irreducible curve contained
in the irreducible component.
This gives a family of rational maps 
$X \dasharrow Y'_{t} \subset \mathbb P^{M}$ of finite degree $d>1$.
After removing finitely many points if necessary
(not the two given points), we may assume
that every $Y'_{t}$ is of general type,
by theorem \ref{generaltype}.
Since the degree of maps is constant
in the deformation, then
the two given maps are birationally equivalent, 
by corollary \ref{globalrigidity}.
\end{proof}

\subsection{Towards an effective estimate}
The theorem above also leads to an almost effective result.

\begin{prop} \label{effective}
In theorem {\rm \ref{finiteness}}
the number of equivalence classes of rational maps 
has an upper bound of the form
$B(n, v_{X})$
only depending on the dimension $n$
and the canonical volume $v_{X} = {\rm vol}(K_{X})$.
Here the function $B$ can be explicitely computed
in terms of the function $r_{n}$ that is
defined in theorem $\ref{HKT}$.
\end{prop}

\begin{proof}

This is obtained  using
a general estimate (of B\'{e}zout type)
for the number of irreducible components
of an algebraic subset.
If $V\subset\mathbb P^{C}$ is defined by equations of degree $\leq D$,
the number of irreducible components of $V$ 
is bounded above by $D^{C}$. The same estimate holds more generally
for a locally closed subset in $\mathbb P^{C}$
of the form $U \cap V$, where $U$ is a Zariski open
subset and $V$ is an algebraic subset, defined by equations of degree $\leq D$.
This is applied to
the  disjoint union of varieties that is defined in theorem \ref{finite number}.

The following will be shown later on in \S  \ref{degreesofequations}:

\begin{itemize}
\item
the embedding dimension of $\chow_{n,\gamma}(X' \times \mathbb P^{M})$
is bounded by some function $C(n,\gamma,M)$,
\item
the algebraic subset $G_{\gamma,k,d}(X' \times \mathbb P^{M})$
is defined by equations whose
degree is bounded by some function $D(n,\gamma,M,d')$,
\end{itemize}
where $d'$ may be taken to be  $\deg X'$, and
moreover these bounds are effectively computable.

It follows that
the number of irreducible components of 
$G_{\gamma,k,d}(X' \times \mathbb P^{M})$ is bounded upperly 
by $D(n,\gamma,M,d')^{C(n,\gamma,M)}$.
The sum over $\gamma,k,d$  is bounded by
$$(\gamma_{X})^3\, D(n,\gamma_{X},M_{X},d')^{C(n,\gamma_{X},M_{X})}.$$
Finally recall that $M_{X}$ and $\gamma_{X}$ and $d'$ are bounded in terms of the 
canonical volume $v_{X}$ and the function $r_{n}$,
see \S \ref{plumap}.
\end{proof}

\begin{rem} \em
The result of Heier \cite{H},
although restricted to morphisms,
is really effective, based on some effective
result of uniform pluricanonical veryampleness,
of the same author. Concerning
the effectivity of the function $r_{n}$,
that is still unknown, see Hacon-McKernan \cite{HK}
and Chen \cite{Chen}.
\end{rem}

\section{Remarks on the complexity of Chow varieties}
\label{chowcomplexity}

In this section we give the proofs of two points
that were used in the proof of proposition \ref{effective}.
The complete details are indeed so cumbersome, and we
prefer to concentrate on the method, especially on a few
points involving the use of Chow varieties, rather than
working out explicit formulas. 
Moreover we believe that the present results
will  be improved and simplified if  a different 
parametrization is used for rational maps,
that we plan to study in a subsequent paper.

\subsection{On elimination of variables}
\label{eliminazione}

As a basic tool we need some estimate of
the complexity of the process of
elimination of variables. This follows from known estimates
of the complexity of Gr\"obner bases of a polynomial ideal,
since such a basis always contains a basis
of the resultant ideal (also called the eliminant ideal).
We refer to the paper of Dub\'e \cite{D}
for the explicit results. Thus the following holds.
\medskip

Let $f_1,\ldots,f_p$ be homogeneous polynomials
of degree $\leq d$  in the variables
$x_0,\ldots,x_r$  with indeterminate coefficients. 
Then the resultant ideal, after elimination of the variables $x_i$,  is
generated by homogeneous polynomials of degree $\leq \delta(r,d)$ in
the coefficients of the polynomials $f_j$, where $\delta$
is some suitable integer function, effectively computable.
\medskip

The following corollaries are easily deduced.
\medskip

If the indeterminate coefficients of  $f_1,\ldots,f_p$ 
are specialized as homogeneous polynomials
of degree $\leq d$ in a new set of variables $y_1,\ldots,y_r$, 
then the resultant ideal, after elimination of the variables $x_i$,  is
generated by homogeneous polynomials of degree 
$\leq \delta'(r,d) := d\, \delta(r,d)$ in
the variables $y_j$.

If $f_1,\ldots,f_p$ are multi-homogeneous polynomials 
of degree  $\leq d$ depending on sequences
$x_1,  \ldots,x_k$  of $r+1$ variables each, with indeterminate coefficients,
then the resultant ideal, after elimination of $x_1,\ldots,x_k$,
is generated by homogeneous polynomials 
in the coefficients of $f_1,\ldots,f_p$, of degree 
at most equal to
$$\delta(k,r,d) := \delta'(r,\delta'(r,\ldots, \delta'(r,d)) 
\underset{k}{\ldots}) ).$$

\subsection{Chow varieties and their equations}

Standard references are \cite{Chow} or \cite{Kol}.
Let $Z$ be a subvariety in $\Bbb P^r=\Bbb P(V)$
of dimension $n$ and degree $k$.
We denote with the symbol $V^\vee$ the dual vector space.
There is an irreducible polynomial $F_Z(u_0,\ldots,u_n)$,
homogeneous of degree $k$  with respect to each variable 
$u_i\in V^{\vee}$, such that $F_Z(u_0,\ldots,u_n) = 0$  if and only if
the linear space $u_0(x)=\cdots=u_n(x)=0$ meets $Z$. This polynomial
$F_Z$, which is  unique up to proportionality, is called the {\em associated form}
of the variety $Z$, and its  coefficients are said to be the coordinates of $Z$. The
associated form of a positive cycle $Z=\sum a_iZ_i$ of pure dimension $n$ is defined as
$F_Z=\prod F_{Z_i}^{a_i}$. 

Let  $\Bbb F_{n,k,r}$ be the projective space of multihomogeneous
polynomials of degree $k$ in $u_0,\ldots,u_n$. 
The dimension of this space is a function
$$\varphi(n,k,r)$$
that is effectively computable, by elementary algebra.

If $S \subset \Bbb P^r$ is an algebraic subset,
then in  $\Bbb F_{n,k,r}$
the subset $\chow_{n,k}(S)$
of associated forms $F_Z$ of cycles
$Z$ of dimension $n$ and degree $k$ supported in $S$ is an
algebraic subset, called the {\em Chow variety} of $S$,
relative to the pair $n,k$.
Equations for the Chow variety are obtained from
the following characterization of associated forms.
Assume that $S$ in $\mathbb P^r$ is 
defined by equations $f_\alpha(x)=0$, of degree at most $d'$.

Necessary and sufficient conditions for
a multihomogeneous form  $F(u_0,u_1,\ldots,u_n)$
so that it is associated to some $n$-cycle
supported in $S \subset \mathbb P^{r}$ are: 

\begin{enumerate}
\item[]  for every $u_1,\ldots,u_n \in V^\vee$ 
there are $x_1,\ldots,x_k \in V$ such that: 
\item  $F(u_0,u_1,\ldots,u_n)$ \ and \ $u_0(x_1)\cdots u_0(x_k)$ \
           are proportional, as polynomials in $u_0$;  
\item $f_\alpha(x_j)=0$ \ for $j=1,\ldots,k$ and for every $\alpha$;
\item  $u_i(x_j)=0$ \  for $i=1,\ldots,n$ and $j=1,\ldots,k$; 
\item  for every $x_j$ and every $v_0,v_1,\ldots,v_n$ passing through
          $x_j$ one has \newline   $F(v_0,v_1,\ldots,v_n)=0$. 
\end{enumerate} 

We now describe
how these conditions can be translated into equations.
The proportionality condition  1 
is  bilinear in the two polynomials.
Every linear form passing through a point $x$ can be expressed as
$x\cdot s:=s(x,-)$ where $s$ is an  antisymmetric bilinear form.
Writing then
$F(x\cdot s_0,x\cdot s_1,\ldots,x\cdot s_n)=\sum\; \varphi_{\alpha}(F,x) 
\;M_{\alpha}(s_0,s_1,\ldots,s_n)$, 
condition 4 turns out to be equivalent to:
$\varphi_{\alpha}(F,x_j)=0$ for every $j$ and every $\alpha$.
The polynomials $\varphi_{\alpha}$ are linear in $F$ 
and of degree $\bar k=k(n+1)$
with respect to the variable $x$. 
The degrees of equations are collected in the following table.
$$\begin{array}{l|cccc}
     & F & u_1,\ldots,u_n & x_1,\ldots,x_k \\ \cline{1-4}
1  & 1 & k,\ldots,k & 1,\ldots,1 \\
2  & - & - & d' \\
3  & - & 1 & 1 \\
4  & 1 & - & \bar k,\ldots, \bar k
\end{array}$$
 
Eliminating  $x_1,\ldots,x_k$ from the equations above
yields some finite system of equations of the form
$P(F, u_1,\ldots,u_n)=0$. 
It follows from \S \ref{eliminazione} that the function
$$\Delta(n,k,r,d') := \delta(k,r,\max \{d', k(n+1)\})$$ 
is an upper bound for the degrees of 
polynomials $P$ with respect to $F$. 

The conditions so that $F$ is an associated form are therefore
equivalent to requiring that: 
for every $u_1,\ldots,u_n$ the sequence $F,\, u_1,\ldots,u_n$
satisfies the resultant equations $P$. Writing
$P(F, u_1,\ldots,u_n) = \sum\; P_{\alpha}(F)\, M_{\alpha}(u)$,
we obtain equations 
$P_{\alpha}(F)=0$  for every $\alpha$ and every $P$,
which define the Chow variety.
Note that the degree of a polynomial $P_{\alpha}$ coincides with
the degree of the corresponding $P$ with respect to the variable $F$. 
Hence all these degrees are bounded by the same function
$\Delta(n,k,r,d')$.

\subsection{On Chow forms of graphs and direct images}
\label{degreesofequations}

We now work with cycles $\Gamma$ 
in a product variety $X' \times \mathbb P^M \subset
\mathbb P^N \times \mathbb P^M$.
Define $\overline M := (N+1)(M+1)-1$,
the dimension $N$ being fixed.
Let $d'$ be an integer such that 
$X'$ admits equations of degree at most $d'$,
for instance  $\deg X'$ is such an integer.
Then $X' \times \mathbb P^{M}$
also admits equations of degree at most $d'$. 
Because of the analysis in the preceding section,
the Chow variety $\chow_{n,\gamma}(X' \times \mathbb P^M)$
is embedded in a projective space of dimension
$$C(n,\gamma,M) := \varphi(n,\gamma, \overline M)$$
and admits equations of degree bounded by
$$D(n,\gamma,M,d') := \Delta(n,\gamma, \overline M,d').$$
We describe how the Chow form of $p_{*}(\Gamma)$
depends on the Chow form of $\Gamma$. 

In the product
$(\mathbb P^{N} \times \mathbb P^{M}) 
\times (\mathbb P^{M})^{*} \times 
\underset{n+1}{\cdots} \times (\mathbb P^{M})^{*}$
consider the incidence subset
$\Phi$ consisting of all sequences
 $(x,y), v_{0}, \ldots, v_{n}$ such that
\begin{itemize}
\item[5.]  $x \otimes y$ belongs to $\Gamma$,
\item[6.]  $v_{i}(y) =0$ \ for $i=0,\ldots,n$,
\end{itemize}
and denote by $q$ the natural projection to
$(\mathbb P^{M})^{*} \times 
\underset{n+1}{\cdots} \times (\mathbb P^{M})^{*}$.
It is easily seen that
$\Phi$ is irreducible of dimension $(n+1)M -1$.
The equations of the subscheme $q(\Phi)$
are obtained by eliminating the variables $x,y$
from the equations above.
The degree of  $q:\Phi \rightarrow q(\Phi)$ is equal to
the degree of  $p: \Gamma \rightarrow Y'$,
provided of course it is a finite degree.
Therefore $q(\Phi)$ is of the same dimension as $\Phi$ if and only if 
$Y'$ is of dimension $n$. In this case the cycle
$q_{*}(\Phi)$
is a hypersurface in 
$(\mathbb P^{M})^{*} \times 
\underset{n+1}{\cdots} \times (\mathbb P^{M})^{*}$,
and is defined by some multihomogeneous polynomial
$G(v_{0}, \ldots, v_{n})$
of the same degree $k$ with respect to each $v_{i}$.
This is the Chow form of $p_{*}(\Gamma)$.

The conditions above may be written as equations
involving the Chow form $F$ of the cycle $\Gamma$ 
and the other variables.
The degrees of these equations are:
$$\begin{array}{l|cccc}
 & F & v_{0},\ldots, v_{n} & x, y \\ \cline{1-4}
\rm 5 & 1 & - & \bar\gamma, \bar\gamma \\
\rm 6 & - & 1 & -,1 \\
\end{array}$$
where $\bar\gamma = \gamma (n+1)$.
Eliminating  $x,y$ from the equations above
one obtains a finite set of equations of the form
$G_\alpha(F,v_{0}, \ldots, v_{n})=0$.
Applying \ref{eliminazione} one sees that
the degrees relative to $F$ are bounded by 
the same function $D(n,\gamma,M,d')$.

The algebraic subfamily $G_{\gamma}(X', \mathbb P^{M})$
of those $\Gamma$ for which $Y'$ still is of dimension $n$
is defined by the condition that all these forms in $v_{0}, \ldots, v_{n}$ 
generate a principal ideal.
In an irreducible component of this subset one has therefore
a single form  $G(F,v_{0}, \ldots, v_{n})$
of the same degree $k$ with respect to each $v_i$,
and whose degree  with respect to $F$
is bounded by $D(n,\gamma,M,d')$.
The algebraic subfamily  $G_{\gamma,k,d}(X', \mathbb P^{M})$ 
of those $F$ such that the form
$G(F) := G(F,-, \ldots, -)$ 
is a $d$-th power  is therefore defined by 
equations arising from the quadratic equations $Q_{\nu}(G)=0$
of the $d$-ple Veronese variety by specialization as
$Q_{\nu}(G(F))=0$,
and are of degree $\leq 2 D(n,\gamma,M,d')$.

We may assume, as a feedback, that the coefficient $2$
was already included in the definition of the function $D$.

\newpage

\noindent {\sc Lucio Guerra }\\
Dipartimento di Matematica e Informatica, Universit\`a di Perugia\\
Via Vanvitelli 1, 06123  Perugia, Italia\\
{\tt guerra@unipg.it}
\bigskip

\noindent {\sc Gian Pietro Pirola}\\
Dipartimento di Matematica, Universit\`a di Pavia\\
via Ferrata 1, 27100 Pavia, Italia\\
{\tt gianpietro.pirola@unipv.it}

\end{document}